\documentclass[12pt]{amsart}

\usepackage{latexsym}
\usepackage{amsmath}
\usepackage{amssymb}
\usepackage{mathtools}

\DeclarePairedDelimiter{\ceil}{\lceil}{\rceil}

\newtheorem{thm}{Theorem}

\newtheorem{lemma}[thm]{Lemma}

\usepackage{epic}
\usepackage{enumerate}

\newcommand{\wt}{\operatorname{wt}}
\newcommand{\LL}{\mathbb{L}}
\newcommand{\CC}{\mathbb{C}}
\newcommand{\ZZ}{\mathbb{Z}}
\newcommand{\QQ}{\mathbb{Q}}
\newcommand{\HH}{\mathbf H}
\newcommand{\Var}{\operatorname{Var}}
\newcommand{\MHS}{\operatorname{MHS}}
\newcommand{\hsp}{\operatorname{hsp}}
\newcommand{\mon}{\operatorname{mon}}

\newcommand{\spec}{\operatorname{spectrum}}
\newcommand{\Sp}{\operatorname{Sp}}

\begin{document}

\title[Motivic Milnor fiber of a plane curve]{A recursive formula for the motivic Milnor fiber of a plane curve}

\author[{M.\ Gonz\'{a}lez Villa}]{Manuel Gonz\'{a}lez Villa}
\address{Department of Mathematics, University of Wisconsin-Madison, 480 Lincoln Drive, Madison, WI 53706, USA}
\email {villa@math.wisc.edu}

\author[G.\ Kennedy]{Gary Kennedy}
\address{Ohio State University at Mansfield, 1760 University Drive,
Mansfield, Ohio 44906, USA}
\email{kennedy@math.ohio-state.edu}

\author[L.\ McEwan]{Lee J. McEwan}
\address{Ohio State University at Mansfield, 1760 University Drive,
Mansfield, Ohio 44906, USA}
\email{mcewan@math.ohio-state.edu}

\begin{abstract}
We find a recursive formula for the
motivic Milnor fiber of an irreducible plane curve, 
using the notions of a truncation and
derived curve.
We then apply natural transformations
to obtain a similar recursion for the Hodge-theoretic spectrum.
\end{abstract}

\maketitle


\section{Introduction}\label{intro}
In this paper we develop a recursive formula for the
motivic Milnor fiber of a singular point of a plane curve,
expressed in terms of the essential exponents in
its Puiseux expansion. This subject has been treated in
\cite{Guibert}, and later in \cite{GPGV}, but the novel feature here is that we develop
a recursion at the level of motives,
adapting the topological techniques of \cite{Kennedy-McEwan}. Alternative approaches to the computation of the motivic Milnor fibre of a singular point of a plane curve proposed in the literature are to use an embedded or log resolution of the singular point \cite{Denef-LoeserBarca}, the Newton process \cite{ACNLM-AMS}, or splicing \cite{Cauwbergs}.
\par
After developing the recursion, we apply natural transformations
to obtain similar recursions for the Hodge-theoretic spectrum
and the topological monodromy.
\par
\emph{Acknowledgements.}
This paper benefited greatly from extensive conversations with Mirel Caib\u{a}r;
we also received helpful advice from Fran\c{c}ois Loeser.
The first author was supported by 
Spanish Ministerio de Ciencia y Tecnolog\'{\i}a Grant no. MTM2013-45710-C2-02-P.
The second author was supported by a Collaboration
Grant from the Simons Foundation, and worked
on this project while in residence at the Fields Institute.

\section{The theorem} \label{thetheorem}
Consider an irreducible complex plane curve $C$ with Puiseux expansion
\begin{equation*}
\zeta=\sum c_{\mu}x^{\mu}
\end{equation*}
in which the essential exponents are
$\mu_1 < \mu_2 < \dots < \mu_e$.
Let $\mu_1 = n/m$, where $m$ and $n$ are relatively prime. 
If necessary we will change coordinates so that 
none of the exponents
are integers and so that $c_{\mu_1}=1$.
Then
$C$ is defined by the vanishing of a function
which can be written both as a product and a sum:
\begin{equation} \label{fproduct}
f(x,y) =\prod (y-\zeta) = (y^m - x^n)^{d'} + \cdots.
\end{equation}
Here $d'$ is some positive integer, and the unnamed terms
have higher order with respect to the weighting $\wt(x)=m$, $\wt(y)=n$;
the product is taken over all possible conjugates,
the number of which is
$d=m d'$.
\par
As in \cite{Kennedy-McEwan}, we define two associated curves.
The  \emph{truncation} $C_1$ is the curve with Puiseux expansion 
$\zeta_1 = x^{\mu_1}$ and defining function $f_1(x,y) = y^m - x^n$; it has $m$ conjugates.
The \emph{derived curve} $C'$ has Puiseux expansion 
\begin{equation} \label{derivedpuiseux}
\zeta'= 
\sum_{\mu>\mu_1} c_{\mu}x^{\mu'},
\end{equation}
where $\mu' = m(\mu - \mu_1 + n)$.
It has $d'$ conjugates, and its defining function is
\begin{equation}\label{fprimeproduct}
f'(x,y)=\prod (y-\zeta')=y^{d'}+\dots,
\end{equation}
a product using each conjugate once.
Its essential exponents are 
$$\mu_1',  \mu_2', \dots, \mu_{e-1}'$$ where
$$\mu_i' = m(\mu_{i+1} - \mu_1 + n).$$ 
\par
Denote by ${\hat\mu}$ the inverse limit of the groups of roots of unity
(not to be confused with the notation for essential exponents).
Following Denef and Loeser \cite{Denef-LoeserBarca},
let 
$K_0^{\hat\mu}(\Var_{\CC})$
denote the 
monodromic Grothendieck ring of varieties with good ${\hat\mu}$-action,
and let 
$\mathcal{M}_{\CC}^{\hat\mu}$
be the localization obtained by inverting $\LL$,
the class of the affine line.
The class of a variety $X$ in this ring will be denoted $[X]$
and will be called a \emph{motive}.
\par
In the space $\mathcal{L}_N(\mathbb{C}^2_{x,y})_\mathbf{0}$ of $N$-jets based at the origin, let 
$\mathcal{X}_{N,1}(f)$ be the subspace consisting of jets
$\varphi$ for which
$$
(f \circ \varphi)(t)=t^N+\text{ higher order terms}.
$$
The \emph{local motivic zeta function} of $f$ at the origin is
$$
Z(f) = \sum_{N=1}^{\infty}[\mathcal{X}_{N,1}(f)]\LL^{-2N}T^N.
$$
The \emph{motivic Milnor fiber} of $f$ at the origin
is the limit
$$
S(f) = -\lim_{T\to\infty} Z(f).
$$
(In fact $Z(f)$ is known to be a rational function; thus we are evaluating at infinity.)
In Section \ref{proofof1} we prove the following recursive formula.
\begin{thm} \label{curvetheorem}
$S(f) = S((f_1)^{d'}) + S(f')-[\mu_{d'}]$.
\end{thm}
\par
In this formula $[\mu_{d'}]$ denotes the motive of roots of unity.
To make the base case of this recursion more explicit,
we introduce the following notation:
$[(f_1)^{d'}-1]$ represents the motive associated to the variety
$
\{ (x,y) \in \mathbb{C}^2 | (f_1(x,y))^{d'}=1 \}.
$
\begin{thm} \label{basecase}
$S((f_1)^{d'}) = [(f_1)^{d'}-1] - [\mu_{d'}](\LL-1)$.
\end{thm}
This formula is implicit in the results of Section 3 of Guibert \cite{Guibert},
but we prefer to give
a short self-contained proof in Section \ref{basecaseproof} below.
Combining Theorems \ref{curvetheorem} and \ref{basecase},
we obtain the following formula:
\begin{equation} \label{combined}
S(f) = [(f_1)^{d'}-1] - [\mu_{d'}]\LL  + S(f').
\end{equation}
As a consequence of (\ref{combined}), the motivic Milnor fiber of $f$ at the origin $S(f)$ is determined by the essential exponents of the complex plane curve $C$.


\section{Proof of Theorem \ref{curvetheorem}} \label{proofof1}
We begin with two basic properties of the motivic Milnor fiber.

\begin{lemma} \label{yalone}
For the coordinate function $y$, we have $S(y^{d'})=[\mu_{d'}].$
\end{lemma}

\begin{lemma} \label{unitinlr}
If the quotient $g/h$ is a
unit in the local ring at the origin, then 
$[\mathcal{X}_{N, 1}(g)]=[\mathcal{X}_{N, 1}(h)]$ for each $N$.
\end{lemma}

Lemma \ref{yalone} follows from a direct calculation. To prove Lemma 
\ref{unitinlr}, observe that if $g=uh$ then $\varphi \mapsto u \,\varphi$
gives an equivariant identification of $\mathcal{X}_{N, 1}(g)$ and
$\mathcal{X}_{N, 1}(h)$.
\par
To prove Theorem \ref{curvetheorem}, we
will use a pair of maps $\mathbb{C}^2_{v,w} \to \mathbb{C}^2_{x,y}$ defined as follows:
\begin{gather*}
\pi(v,w)= (v^m,v^n(1+w)),\\
\pi'(v,w)=(v,v^{mn}w),
\end{gather*}
and we will work with jets in $\mathcal{L}_{N}(\mathbb{C}^2_{v,w})_\mathbf{0}$.

\begin{lemma} \label{locunit}
The quotient $\frac{f' \circ \pi'}{f \circ \pi}$ is a unit in the local ring at the origin.
Thus $[\mathcal{X}_{N, 1}(f \circ \pi)]=[\mathcal{X}_{N, 1}(f' \circ \pi')]$ for each $N$.
\end{lemma}
\begin{proof}
\par
Looking at the product in (\ref{fproduct}) and composing
with $\pi$, we observe that we may write
\begin{equation*}
\begin{aligned}
y-\zeta&=v^n(1+w)-\left((v^m)^{n/m}+\sum_{\mu>\mu_1} c_{\mu}(v^m)^{\mu}\right) \\
&=v^n \left(1+w-1^{n/m}-\sum_{\mu>\mu_1} c_{\mu}v^{m\mu-n}
\right).
\end{aligned}
\end{equation*}
The interpretation of $1^{n/m}$ depends on the choice of conjugate;
each possible $m$th root of unity occurs equally often.
Thus in $d'$ of the factors the interpretation is that $1^{n/m}=1$.
In the remaining factors $1^{n/m}$ is some other $m$th root of unity,
and within each of these factors the subfactor within parentheses is a unit.
The full product $f \circ \pi(v,w)$ is thus a unit times
$$
v^{nd} \prod (w-\sum_{\mu>\mu_1} c_{\mu}v^{m\mu-n}),
$$
where the product is taken over all $d'$ conjugates.
\par
Now looking at the product in (\ref{fprimeproduct}) and composing
with $\pi'$, we see that
\begin{equation*}
\begin{aligned}
f' \circ \pi'(v,w)&=\prod (v^{mn}w-\sum_{\mu>\mu_1} c_{\mu}v^{\mu'}) \\
&=v^{mnd'} \prod (w-\sum_{\mu>\mu_1} c_{\mu}v^{\mu'-mn}),
\end{aligned}
\end{equation*}
precisely the same product over all $d'$ conjugates.
\end{proof}
\par
Given a jet $\rho \in \mathcal{L}_N(\mathbb{C}^2_{v,w})_\mathbf{0}$, its \emph{pushforward}
by $\pi$ is the jet $\pi \circ \rho \in \mathcal{L}_N(\mathbb{C}^2_{x,y})_\mathbf{0}$. 
Note that the pushforward of a jet in $\mathcal{X}_{N, 1}(f \circ \pi)$ is a jet in $\mathcal{X}_{N, 1}(f)$. 
The pushforward jets form a subspace which we denote by $\mathcal{X}^{\pi}_{N, 1}(f)$.
In the same way, we define three similar subspaces:
\begin{itemize}
\item
$\mathcal{X}^{\pi}_{N, 1}((f_1)^{d'})$ inside $\mathcal{X}_{N, 1}((f_1)^{d'})$,
\item
$\mathcal{X}^{\pi'}_{N, 1}(f')$ inside $\mathcal{X}_{N, 1}(f')$,
\item
$\mathcal{X}^{\pi'}_{N, 1}(y^{d'})$ inside $\mathcal{X}_{N, 1}(y^{d'})$.
\end{itemize}

\begin{lemma} \label{identpi}
Each jet in $\mathcal{X}^{\pi}_{N, 1}(f)$ is the pushforward of a unique jet in
$\mathcal{X}_{N, 1}(f \circ \pi)$.
Thus the map $\pi$ identifies the space $\mathcal{X}_{N, 1}(f \circ \pi)$ with the space 
$\mathcal{X}^{\pi}_{N, 1}(f)$.
It likewise identifies $\mathcal{X}_{N, 1}((f_1)^{d'} \circ \pi)$ with 
$\mathcal{X}^{\pi}_{N, 1}((f_1)^{d'})$.
The jet spaces $\mathcal{X}_{N, 1}(f) \setminus \mathcal{X}^{\pi}_{N, 1}(f)$ and
$\mathcal{X}_{N, 1}((f_1)^{d'}) \setminus \mathcal{X}^{\pi}_{N, 1}((f_1)^{d'})$ 
obtained by set-theoretic difference
are equal.
\end{lemma}

\begin{proof}
We claim that
$\mathcal{X}^{\pi}_{N, 1}(f)$ consists of all jets in $\mathcal{X}_{N, 1}(f)$ of the form
\begin{equation}\label{specialform}
\varphi(t) = (x_0 t^{mp}+\dots,y_0 t^{np}+\dots),
\end{equation}
where $p$ is  a positive integer, with $x_0$ and $y_0$ being nonzero numbers satisfying  
\begin{equation}\label{cancellation}
(y_0)^m = (x_0)^n.
\end{equation}
Indeed, consider a jet $\rho$ in $\mathcal{X}_{N, 1}(f \circ \pi)$; suppose that
$\rho(t)=(v(t),w(t))$, where $v(t)=v_{0}t^{p}+\dots$, with $v_0\neq 0$. Then
$$
\pi \circ \rho(t)=((v_0)^{m}t^{mp}+\dots,(v_0)^{n}t^{np}+\dots).
$$
Conversely, given
$$
\varphi(t)=(x(t),y(t))=(x_0 t^{mp}+\dots,y_0 t^{np}+\dots)
$$
with $(y_0)^m = (x_0)^n$,
note that there is a unique number $v_0$ for which $x_0=(v_0)^m$ and $y_0=(v_0)^n$. Explicitly, letting $r$ and $s$ be the smallest positive integers for which
$$ \det\left[ \begin{array}{cc} 
	m  & n  \\
	r & s
\end{array} \right]=1,$$
we have $v_0=(x_0)^s/(y_0)^r$. Let $v(t)$ be the $m$th root of $x(t)$ with this leading coefficient, i.e.,
$$
v(t)=v_{0}t^p+\dots \quad \text{and} \quad (v(t))^m=x(t).
$$
Let $w(t)=y(t)/(v(t))^n-1$. Then $\rho(t)=(v(t),w(t))$ gives the unique jet whose pushforward is $\varphi$.
\par
The same argument applies with $(f_1)^{d'}$ in place of $f$, showing that 
$\mathcal{X}^{\pi}_{N, 1}((f_1)^{d'})$ consists of all jets in 
$\mathcal{X}_{N, 1}((f_1)^{d'})$ of the form specified by 
by equations
(\ref{specialform}) and (\ref{cancellation}),
and that each such jet is a pushforward in a unique way.
\par
Again consider a jet $\varphi \in \mathcal{X}^{\pi}_{N, 1}(f)$.
Equations
(\ref{specialform}) and (\ref{cancellation})
imply a cancellation:
there is no term of degree $mnpd'$, and thus we must
have $N>mnpd'$. The same remark applies to a jet in
$\mathcal{X}^{\pi}_{N, 1}((f_1)^{d'})$.
For a jet $\varphi \in \mathcal{X}_{N, 1}(f) \setminus \mathcal{X}^{\pi}_{N, 1}(f)$, 
however, the cancellation just described does not occur. 
This means that  
$(f \circ \varphi)(t)$ and $((f_1)^{d'} \circ \varphi)(t)$ have the same lowest-order term.
Thus the jet spaces $\mathcal{X}_{N, 1}(f) \setminus \mathcal{X}^{\pi}_{N, 1}(f)$ and
$\mathcal{X}_{N, 1}((f_1)^{d'}) \setminus \mathcal{X}^{\pi}_{N, 1}((f_1)^{d'})$ 
coincide.
\end{proof}

\begin{lemma} \label{identpiprime}
Each jet in $\mathcal{X}^{\pi'}_{N, 1}(f')$ is the pushforward of a unique jet in
$\mathcal{X}_{N, 1}(f' \circ \pi').$
Thus the map $\pi'$ identifies the space $\mathcal{X}_{N, 1}(f' \circ \pi')$
with the space $\mathcal{X}^{\pi'}_{N, 1}(f')$.
It likewise identifies 
$\mathcal{X}_{N, 1}(y^{d'} \circ \pi')$ with 
$\mathcal{X}^{\pi'}_{N, 1}(y^{d'})$.
The $\mathcal{X}_{N, 1}(f') \setminus \mathcal{X}^{\pi'}_{N, 1}(f')$ and
$\mathcal{X}_{N, 1}(y^{d'}) \setminus \mathcal{X}^{\pi'}_{N, 1}(y^{d'})$ are equal.
\end{lemma}

\begin{proof}
We claim that
$\mathcal{X}^{\pi'}_{N, 1}(f')$ consists of all jets in $\mathcal{X}_{N, 1}(f')$ of the form
\begin{equation} \label{formofjet}
\varphi(t) =(x(t),y(t))= (x_0 t^{p}+\dots,y_0 t^{q}+\dots),
\end{equation}
(with $x_0\neq 0$ and $y_0\neq 0$), where $q>mnp$,
together with all jets in which $y(t)$ is identically zero.
Indeed, consider a jet $\rho(t)=(v(t),w(t))$ in $\mathcal{X}_{N, 1}(f' \circ \pi')$. Note that
$v(t)$ can't be identically zero, and write it as $v(t)=x_{0}t^p+\dots$. Then its pushforward has the required form. Conversely, given a jet $\varphi(t)=(x(t),y(t))$ of this form, compose with the birational inverse of $\pi'$, i.e., let $v(t)=x(t)$ and $w(t)=y(t)/(x(t))^{mn}$, to obtain the unique jet $\rho$ whose pushforward is $\varphi$.
The same argument applies with $y^{d'}$ in place of $f'$.
\par
Thus a jet $\varphi$ in
$\mathcal{X}_{N, 1}(f') \setminus \mathcal{X}^{\pi'}_{N, 1}(f')$
has the form shown in equation (\ref{formofjet}), with $q \leq mnp$,
and likewise for a jet in $\mathcal{X}_{N, 1}(y^{d'}) \setminus \mathcal{X}^{\pi'}_{N, 1}(y^{d'})$
If we expand the product in (\ref{fprimeproduct}) and then compose each term with 
such a jet, the leading term $y^{d'}$ has order $d'q$. 
Observe that each exponent $\mu'$ appearing in (\ref{derivedpuiseux}) is greater than $mn$.
Thus each other term in the expansion will have order greater than $d'q$ . Thus $f \circ \varphi$
and $y^{d'} \circ \varphi$ have the same leading term.
\end{proof}
\par
Recall that Lemma \ref{locunit} states this equation of motives:
\begin{itemize}
\item
$[\mathcal{X}_{N, 1}(f \circ \pi)]=[\mathcal{X}_{N, 1}(f' \circ \pi')]$
\end{itemize}
Lemma \ref{identpi} implies the following equations:
\begin{itemize}
\item
$[\mathcal{X}^{\pi}_{N, 1}(f)]=[\mathcal{X}_{N, 1}(f \circ \pi)]$
\item
$[\mathcal{X}^{\pi}_{N, 1}((f_1)^{d'})]=[\mathcal{X}_{N, 1}((f_1)^{d'} \circ \pi)]$
\item
$[\mathcal{X}_{N, 1}(f)]-[\mathcal{X}^{\pi}_{N, 1}(f)]
=[\mathcal{X}_{N, 1}((f_1)^{d'})]-[\mathcal{X}^{\pi}_{N, 1}((f_1)^{d'})]$
\end{itemize}
Lemma \ref{identpiprime} gives us the following equations:
\begin{itemize}
\item
$[\mathcal{X}^{\pi'}_{N, 1}(f')]=[\mathcal{X}_{N, 1}(f' \circ \pi')]$
\item
$[\mathcal{X}^{\pi'}_{N, 1}(y^{d'})]=[\mathcal{X}_{N, 1}(y^{d'} \circ \pi')]$
\item
$[\mathcal{X}_{N, 1}(f')]-[\mathcal{X}^{\pi'}_{N, 1}(f')]
=[\mathcal{X}_{N, 1}(y^{d'})] - [\mathcal{X}^{\pi'}_{N, 1}(y^{d'})]$
\end{itemize}
Thus
\begin{equation*}
\begin{aligned}
-[\mathcal{X}_{N, 1}(f)]
=&
-[\mathcal{X}_{N, 1}((f_1)^{d'})]
-[\mathcal{X}_{N, 1}(f')] \\
&+[\mathcal{X}_{N, 1}((f_1)^{d'} \circ \pi)]
+[\mathcal{X}_{N, 1}(y^{d'})] - [\mathcal{X}_{N, 1}(y^{d'} \circ \pi')].
\end{aligned}
\end{equation*}
\par
Using each term as coefficient in a power series $\sum a_{N}\LL^{-2N}T^N$
and taking the limit as $T$ approaches infinity, we obtain this equation:
\begin{equation}\label{assembled}
\begin{aligned}
S(f)
=\,& S((f_1)^{d'}) + S(f') \\
&
- S((f_1)^{d'} \circ \pi)) - S(y^{d'}) + S(y^{d'} \circ \pi')
\end{aligned}
\end{equation}
Both $(f_1)^{d'} \circ \pi(v,w)$ and $y^{d'} \circ \pi'(v,w)$
are $v^{mnd'}w$ times a unit; Lemma \ref{unitinlr} tells us 
that $S((f_1)^{d'} \circ \pi)=S(y^{d'} \circ \pi)$.
Thus by applying Lemma \ref{yalone}, we obtain the 
statement of
Theorem \ref{curvetheorem}.
\par


\section{Proof of Theorem \ref{basecase}} \label{basecaseproof}
To prove Theorem \ref{basecase},
we will again use the decomposition of the jet space 
$\mathcal{X}_{N, 1}((f_1)^{d'})$
into
$ \mathcal{X}^{\pi}_{N, 1}((f_1)^{d'})$
plus its complement.
The theorem is an immediate consequence of these two equations:
\begin{gather}
-\lim_{T\to\infty} \sum_{N=1}^{\infty}
[\mathcal{X}^{\pi}_{N, 1}((f_1)^{d'})]
\LL^{-2N}(f))T^N = - [\mu_{d'}](\LL-1) \label{pushcontribution}
\\
-\lim_{T\to\infty} \sum_{N=1}^{\infty}
[\mathcal{X}_{N, 1}((f_1)^{d'}) \setminus \mathcal{X}^{\pi}_{N, 1}((f_1)^{d'})]
\LL^{-2N}(f))T^N = [(f_1)^{d'}-1] \label{othercontribution}
\end{gather}
\par
To prove equation (\ref{pushcontribution}),
we first invoke
Lemma \ref{identpi} to replace 
$\mathcal{X}^{\pi}_{N, 1}((f_1)^{d'})$ by
$\mathcal{X}_{N, 1}((f_1)^{d'} \circ \pi)$,
and
remark that
$$
(f_1)^{d'} \circ \pi(u,v)= (v^{mn}(mw+\dots))^{d'}.
$$
Consider a jet $\rho$ in $[\mathcal{X}_{N, 1}((f_1)^{d'} \circ \pi)]$, and write it
as
$$
\rho(t)=(v_{p}t^p+\dots,w_{q}t^q+\dots),
$$
with nonzero $v_p$ and $w_q$. Then $N=d'(mnp+q)$ and  $(mv_{p}w_{q})^{d'}=1$.
Thus
$$
\begin{aligned}
\sum_{N=1}^{\infty} [\mathcal{X}_{N, 1}((f_1)^{d'} \circ \pi)]&\LL^{-2N}T^N
=\sum_{p=1}^{\infty} \sum_{q=1}^{\infty}[\mu_d'](\LL-1)\LL^{N-p}\LL^{N-q}\LL^{-2N}T^N \\
&=[\mu_d'](\LL-1)\sum_{p=1}^{\infty} \LL^{-p}T^{d'mnp}
\sum_{q=1}^{\infty}\LL^{-q}T^{d'q} \\
&=[\mu_d'](\LL-1) \cdot\frac{\LL^{-1}T^{d'mn}}{1-\LL^{-1}T^{d'mn}}
\cdot\frac{\LL^{-1}T^{d'}}{1-\LL^{-1}T^{d'}},
\end{aligned}
$$
which yields the desired formula.
\par
To prove equation (\ref{othercontribution}),
we examine a jet 
$$
\varphi \in \mathcal{X}_{N, 1}((f_1)^{d'}) \setminus \mathcal{X}^{\pi}_{N, 1}((f_1)^{d'}).
$$
As we remarked in the proof of Lemma \ref{identpi},
for such a jet we do not see the cancellation
of terms implied by equations (\ref{specialform}) and (\ref{cancellation}).
\par
Suppose that the order of vanishing of $(f_1)^{d'}\circ \varphi$
is divisible by $md'$ but not $mnd'$.
Then $\varphi$ must have the following form:
$$
\varphi(t) =(x(t),y(t))= (x_0 t^{p}+\dots,y_0 t^{q}+\dots),
$$
where the following conditions are satisfied:
\begin{itemize}
\item
$q$ is not a multiple of $n$,
\item
$(y_0)^{md'}=1$,
\item
$p=\ceil{\frac{mq}{n}}$.
\end{itemize}
Note that $x_0$ may vanish; in fact $x(t)$ may even be identically zero.
The order of vanishing of $(f_1)^{d'}\circ \varphi$ is $N=mqd'$.
The motive associated to all such jets is
$$
[\mu_{md'}]\LL^{2mqd'-p-q+1},
$$
and thus its contribution to the sum in (\ref{othercontribution}) is
$$
[\mu_{md'}]\LL^{-p-q+1}T^{mqd'}.
$$
Observe that if we increase $q$ by $n$,
then $p$ increases by $m$
and the order of vanishing increases by $mnd'$.
Thus the contributions of these jet spaces
to the sum in (\ref{othercontribution}) may be packaged
as $n-1$ geometric series
with common ratio $r=\LL^{-(m+n)}T^{mnd'}$.
For each series, the leading term is determined 
by a value of $q$ that is less than $n$, and thus
the order of vanishing is less than $mnd'$.
Hence the contribution of this series to the limit
as $T \to \infty$ is zero.
A similar argument shows that there is no contribution
from jets for which the order of vanishing of $(f_1)^{d'}\circ \varphi$
is divisible by $nd'$ but not $mnd'$.
\par
It remains to analyze those jets
for which the order of vanishing $N$ of $(f_1)^{d'}\circ \varphi$
is divisible by $mnd'$.
Such a jet has the form given in equation (\ref{specialform}),
where
\begin{gather*}
N=mnpd' \\
\text{and } ((y_0)^m-(x_0)^n)^{d'}=1;
\end{gather*}
note that $x_0$ or $y_0$ may vanish.
The motive associated to such jets is
$$
[(f_1)^{d'}-1]
\LL^{2mnpd'-(m+n)p},
$$
and thus the total contribution of these jets to the sum
in equation (\ref{othercontribution}) is
\begin{equation*}
[(f_1)^{d'}-1]\sum_{p=1}^{\infty} \LL^{-(m+n)p}T^{mnpd'} 
=[(f_1)^{d'}-1]\frac{\LL^{-(m+n)T^{mnd'}}}{1-\LL^{-(m+n)T^{mnd'}}}.
\end{equation*}
Taking the limit as $T$ goes to infinity gives us $-[(f_1)^{d'}-1]$.
\par


\section{A recursive formula for the spectrum}\label{specrecursion}
Steenbrink \cite{Steenbrink-Asterisque} and Saito \cite{Saito-MA91} associate to a plane curve
singularity a multiset of rational numbers called its \emph{spectrum} --- it can be
regarded as an element of the group ring  $\ZZ[t^{1/\ZZ}]$.
As Steenbrink \cite{Steenbrink85} says, it ``gathers the information about the eigenvalues of the monodromy
operator \dots and about the Hodge filtration on the vanishing
cohomology." We briefly explain this construction.
\par
We have defined the motivic Milnor fiber $S(f)$ as 
an element of 
$\mathcal{M}_{\CC}^{\hat\mu}$;
the class $1-S(f)$ is called the \emph{motivic vanishing cycle}.
There is a natural map $\chi_h^{\mon}$, called the \emph{Hodge characteristic},
from $\mathcal{M}_{\CC}^{\hat\mu}$
to $K_0(HS^{\mon})$, the Grothendieck ring of mixed Hodge structures
with a finite automorphism.
For a variety $X$ with good ${\hat\mu}$-action,
we associate the sum
$$
\sum_i (-1)^{i} [H_c^i(X,\QQ)] \in K_0(HS^{\mon}).
$$
Here $H_c^i(X,\QQ)$ indicates the $i$th cohomology group with compact supports,
which is naturally endowed with a mixed Hodge structure and a finite automorphism.
This map passes to the localization; thus we have a linear map 
$$
\chi_h^{\mon}: \mathcal{M}_{\CC}^{\hat\mu}
\to K_0(HS^{\mon}).
$$
Steenbrink \cite{Steenbrink76} and Varchenko \cite{Varchenko}
showed that
one can define a natural mixed Hodge structure $\MHS(f)$
on the cohomology of the classical Milnor fiber, which can be regarded as an element
of $K_0(HS^{\mon})$. In Theorem 3.5.5 of \cite{Denef-LoeserBarca},
Denef and Loeser prove that
the Hodge characteristic of the motivic Milnor fiber $S(f)$ is $\MHS(f)$. 
\par
There is also a natural linear map from $K_0(HS^{\mon})$ to $\ZZ[t^{1/\ZZ}]$,
called the \emph{Hodge spectrum}.
For a mixed Hodge structure $H$, let $H_{\alpha}^{p,q}$ denote the 
eigenspace of $H^{p,q}$ associated to the eigenvalue $\exp(2\pi i \alpha)$.
We define 
$$
\hsp(H) = \sum_{\alpha \in \QQ \cap [0,1)}
t^{\alpha}
\,\sum t^p \dim(H_{\alpha}^{p,q}).
$$
Composing with the Hodge characteristic, we obtain a map 
$$
\Sp: \mathcal{M}_{\CC}^{\hat\mu} \to \ZZ[t^{1/\ZZ}].
$$
Denef and Loeser \cite{Denef-LoeserBarca} show that the spectrum of $f$
(as defined by Steenbrink \cite{Steenbrink-Asterisque} and Saito \cite{Saito-MA91})
can be computed using
$$
\spec(f) = \Sp(1-S(f)).
$$
\par
Applying $\Sp$ to the formula of Theorem 1, we obtain this recursive
formula for the spectrum:
\begin{equation} \label{spectrumrecursion1}
\spec(f) = \spec((f_1)^{d'}) + \spec(f') + \Sp([\mu_{d'}] -1).
\end{equation}
As an alternative, here is the result of applying $\Sp$
to formula (\ref{combined}):
\begin{equation} \label{spectrumrecursion2}
\spec(f) =
-\Sp([(f_1)^{d'}-1]) + \Sp([\mu_{d'}]\LL) 
+ \spec(f').
\end{equation}
\par
To make this a usable recursion, we supplement formula (\ref{spectrumrecursion2})
with results from Guibert \cite{Guibert}.
In Lemme 3.4.2(ii), he tells us that
$$
\Sp([f_1-1]) = t - \frac{t^{1/m}-t}{1-t^{1/m}}\cdot \frac{t^{1/n}-t}{1-t^{1/n}}.
$$
Thus, as a multiset, it consists of $(m-1)(n-1)/2$
numbers in the interval $(0,1)$ and an equal number in the interval $(1,2)$
(all of these counting negatively), together with the number 1.
Furthermore we remark that the multiset is symmetric
with respect to reflection across the number 1.
We may  write $\Sp([f_1-1])$ as a sum 
$\Sp^{(0)}([f_1-1])+\Sp^{(1)}([f_1-1])$,
using first the terms with exponents between 0 and 1, and then
those with exponents greater than or equal to 1.
Guibert's Lemme 3.4.3 says that 
\begin{equation} \label{3.4.3}
\begin{aligned}
\Sp  ([(f_1)^{d'}&-1]) (t)
=\frac{1-t}{1-t^{1/d'}} \cdot \\
& \left(
\Sp^{(0)}([f_1-1])(t^{1/d'})
+
t^{1-1/d'}\Sp^{(1)}([f_1-1])(t^{1/d'})
\right).
\end{aligned}
\end{equation}
\par
We claim that the  term 
$\Sp([\mu_{d'}]\LL)$
in formula (\ref{spectrumrecursion2}) is
$$
\frac{1-t}{1-t^{1/d}}\cdot t.
$$
Indeed, the action on the product $\mu_{d'} \times \CC$
has as eigenvalues the $d'$-th roots of unity, and
the cohomology lives in $H^{1,1}$.
Now note that this term cancels some of the contributions
to formula (\ref{3.4.3});
this has the same effect as if the initial term $t$ had been omitted
from $\Sp^{(1)}([f_1-1])$.
Thus again the spectrum has reflectional symmetry.
The upshot is that the contribution of the first two terms on the left of 
(\ref{spectrumrecursion2}) to the spectrum of $f$ is obtained from
 the spectrum of $f_1$ by the following process: 
 \begin{itemize}
 \item
 Discard the spectral numbers greater than 1.
 \item
 Compress by a factor of $1/d'$, i.e., multiply each spectral number by this value.
 \item
 Take the union of $d'$ copies: use the first copy unaltered, add $1/d'$ to each
 spectral number to get the second copy, etc.
\item
 To obtain the remaining contributions in the interval $(1,2)$, reflect across 1.
 \end{itemize}
 \par
Finally we remark that formula (\ref{spectrumrecursion1}) and the auxiliary formulas
imply formula (6) of Theorem 2.3 in \cite{Kennedy-McEwan}, a recursion for the monodromy:
\begin{equation}\label{monorecursion}
\HH(t)=
\frac{\HH_1(t^{d'})\cdot \HH'(t)}
{t^{d'}-1}.
\end{equation}
To see this, observe that 
formula (\ref{3.4.3}) implies that the eigenvalues of
the monodromy for $(f_1)^{d'}$ are the $d'$th roots of the eigenvalues of 
the monodromy for $f_1$. 
Combining this with formula (\ref{spectrumrecursion1}) yields the result.

\section{An example}
\par
Let us consider Example 2.2 from \cite{Kennedy-McEwan}. We begin with the curve whose Puiseux expansion is
$$\zeta = x^{3/2} + x^{7/4} +x^{11/6}.$$
Then its truncation $f_1= y^2 -x^3$ is parametrized by $\zeta_1=x^{3/2}$, and its derived curve is parametrized by 
$$\zeta'=x^{13/2} + x^{20/3}.$$
Repeating this process, we obtain a truncation $f'_1= y^2 - x^{13}$, parametri\-zed by $\zeta'_1=x^{13/2}$ and a second derived curve $f''=y^3 - x^{79}$ with parametrization 
$$\zeta''=x^{79/3}.$$
The number of conjugates of the curve, its truncation, the derived curve, the second truncation, and the second derived curve are $d=1$, $d_1=2$, $d'=6$, $d'_1=2$, and $d''=3$. 
\par
Applying our recursive formula (\ref{combined}) repeatedly, the motivic Milnor fiber $S(f)$ equals 
\begin{equation*}
[(y^2-x^3)^6-1] + [(y^2-x^{13})^3-1] + [y^3-x^{79}-1]
- [\mu_6]\LL - [\mu_3]\LL - \LL + 1.
\end{equation*}
\par
Before computing the spectrum, we remark that formula (\ref{monorecursion})
implies a recursion for the Euler
characteristic of the Milnor fiber (formula (4) of the same cited theorem).
Applying that formula here, we infer that the Milnor number is 204;
thus our recursion will yield 204 spectral numbers.
\par
To compute the spectrum of $f$, we first apply the process described in the previous section
to the two spectral numbers $\tfrac{5}{6}$ and $\tfrac{7}{6}$
of $f_1=y^2-x^3$. Using just the spectral number $\tfrac{5}{6}$,
compression and copying gives
these six spectral numbers in the interval $(0,1)$:
$$\tfrac{5}{36}, \tfrac{11}{36}, \tfrac{17}{36}, \tfrac{23}{36}, \tfrac{29}{36}, \tfrac{35}{36}.$$
Similarly, the spectral numbers
$
\tfrac{15}{26}, \tfrac{17}{26}, \tfrac{19}{26}, \tfrac{21}{26}, \tfrac{23}{26}, \tfrac{25}{26}
$
of $f'_1=y^2-x^{13}$
give us eighteen additional spectral numbers of $f$:
$$
\tfrac{15}{78}, \dots \tfrac{25}{78}, \quad \tfrac{41}{78}, \dots \tfrac{51}{78}, \quad \tfrac{67}{78}, \dots \tfrac{77}{78}.
$$
Likewise there are contributions coming from the $78$ spectral numbers of $y^3-x^{79}$ which lie
in the interval $(0,1)$; these contributions are
$$\tfrac{82}{237}, \tfrac{85}{237}, \tfrac{88}{237}, \tfrac{91}{237}, \dots, \tfrac{235}{237}$$
and  
$$\tfrac{161}{237}, \tfrac{164}{237}, \tfrac{167}{237}, \tfrac{170}{237}, \dots, \tfrac{236}{237}.$$
In total this gives 102 spectral numbers less than 1; after reflection we obtain an additional 102 values,
which together account for the Milnor number.


\bibliographystyle{amsplain}

\providecommand{\bysame}{\leavevmode\hbox to3em{\hrulefill}\thinspace}
\providecommand{\MR}{\relax\ifhmode\unskip\space\fi MR }
\providecommand{\MRhref}[2]{%
  \href{http://www.ams.org/mathscinet-getitem?mr=#1}{#2}
}
\providecommand{\href}[2]{#2}

\end{document}